\documentclass[12pt]{amsart}
\usepackage{amssymb,amsmath,amsthm}

\renewcommand{\H}{\mathcal{H}}

\usepackage{fancyhdr}
\newcommand{\I}{\mathcal{I}}
\newcommand{\Ht}{G}
\renewcommand{\L}{\bar L}

\newcommand{\N}{\mathbb N}

\newcommand{\C}{\mathbb C}

\DeclareMathOperator{\Dom}{Dom}

\newcommand{\p}{\partial}

\newcommand{\z}{\bar z}

\newcommand{\dbar}{\bar\partial}
\newcommand{\dbars}{\bar\partial^*}
\newcommand{\dbarb}{\bar\partial_b}
\newcommand{\dbarbs}{\bar\partial_b\hspace{-3.5pt}{}^*}
\newcommand{\vp}{\varphi}

\newcommand{\lam}{\lambda}
\newcommand{\ep}{\epsilon}

\newtheorem{thm}{Theorem}[section]
\newtheorem{prop}[thm]{Proposition}

\newtheorem{defn}[thm]{Definition}
\newtheorem{rem}[thm]{Remark}

\begin{document}
\begin{abstract}Let $\Omega\subset\C^n$ be a bounded smooth pseudoconvex domain. We show that compactness of the complex Green operator $G_{q}$ on
$(0,q)$-forms on $b\Omega$ implies compactness of the $\dbar$-Neumann operator $N_{q}$ on $\Omega$. We prove that if $1 \leq q \leq n-2$ and $b\Omega$ satisfies $(P_q)$ and $(P_{n-q-1})$, then $G_{q}$ is a compact operator (and so is $G_{n-1-q}$). Our method relies on a jump type formula to represent forms on the boundary, and we prove an auxiliary compactness result for an `annulus' between two pseudoconvex domains. Our results, combined with the known characterization of compactness in the $\overline{\partial}$-Neumann problem on locally convexifiable domains, yield the corresponding characterization of compactness of the complex Green operator(s) on these domains.
\end{abstract}

\title [Compactness of the complex green operator] 
{Compactness of the complex Green operator}
\author{Andrew S.~Raich \and Emil J.~Straube}

\address{
Department of Mathematics\\ Texas A\&M University\\ Mailstop 3368  \\ College Station, TX
77845-3368}

\thanks{Research supported in part by NSF grant DMS 0500842}

\email{araich@math.tamu.edu, straube@math.tamu.edu}

\maketitle

%
%
\section{Introduction and Results}\label{sec:results}
Let $\Omega\subset\C^n$ be a bounded, smooth pseudoconvex domain. The Cauchy-Riemann operator $\dbar$ is a closed,
densely defined operator mapping $L^2_{(0,q)}(\Omega)\to L^2_{(0,q+1)}(\Omega)$ and satisfying $\overline{\partial}^{2}=0$. The associated complex is the $\overline{\partial}$, or Dolbeault, complex.
Let $\dbars$ be the $L^2$-adjoint of $\dbar$, and $\Box = \dbars\dbar + \dbar\dbars$, the $\dbar$-Neumann Laplacian. When $\Omega$ is pseudoconvex, and $1\leq q \leq n$, $\Box$ acting on $\Dom(\Box) \subset L^2_{(0,q)}(\Omega)$ is invertible with a bounded inverse
$N_q$. This inverse is called the $\dbar$-Neumann operator. We refer the reader to \cite{FoKo72,BoSt99,ChSh01,FuSt01,Str06} for background on the $\overline{\partial}$-Neumann problem and its $L^{2}$-Sobolev theory. 

On $b\Omega$, $\dbar$ induces the tangential Cauchy-Riemann operator  $\overline{\partial}_{b}$. Kohn and Rossi introduced
the $\overline{\partial}_{b}$ complex in an effort to understand the holomorphic extension of $CR$-functions from the boundaries of complex manifolds \cite{KoRo65}. Let $\dbarbs$ be the $L^2$-adjoint of $\dbar_b$, and $\Box_b = \dbarb\dbarbs + \dbarbs\dbarb$, the Kohn Laplacian. When $0\leq q \leq n-1$, $\Box_b$ is invertible (on $(\ker\dbarb)^{\perp}$ when $q=0$, and on $(\ker\dbarbs)^{\perp}$ in the case $q=n-1$) with inverse $\Ht_q$. $G_{q}$ is the complex Green operator. In particular, $\overline{\partial}_{b}$, $\overline{\partial}_{b}^{*}$, and $\Box_{b}$ have closed range. Details may be found in \cite{Shaw85,BoSh86,Kohn86,ChSh01}. The regularity and mapping properties of $\dbarb$ are well understood when $\Omega$ is of finite type and satisfies the condition that all eigenvalues of the Levi form are comparable. In this case, optimal subelliptic estimates (so called maximal estimates) were shown in \cite{Koe02}. This work unifies earlier work for strictly pseudoconvex domains and for domains of finite type in $\mathbb{C}^{2}$. We refer the reader to \cite{Koe02} for references to this earlier work and further discussion. For general domains, it is known that subellipticity of $\Ht_q$ implies finite type \cite{Dia86, Koe04}. Global regularity, in the sense of preservation of Sobolev spaces, holds when $\Omega$ admits a defining function that is plurisubharmonic at points of the boundary (\cite{BoSt91}). A defining function is called plurisubharmonic at the boundary when its complex Hessian at points of the boundary is positive semidefinite in all directions. For example, all convex domains admit such defining functions.  

The question we address in this article is that of compactness of the complex Green operator, and, to some extent, its relationship to compactness of the $\overline{\partial}$-Neumann operator.  The results discussed above notwithstanding, the regularity results for $\overline{\partial}_{b}$ do not always parallel those for $\overline{\partial}$. The reason is that there is a symmetry in the form levels for $\overline{\partial}_{b}$ with respect to compactness and subellipticity that is absent for $\overline{\partial}$. This phenomenon was pointed out by Koenig (\cite{Koe04}, p.289). He associates to a $q$-form $u$ on $b\Omega$ an $(n-1-q)$-form $\tilde{u}$ (through a modified Hodge-$*$  construction) such that $\|u\| \approx \|\tilde{u}\|$,  $\overline{\partial}_{b}\tilde{u}=(-1)^{q}\widetilde{(\overline{\partial}^{*}_{b}u)}$, and $\overline{\partial}_{b}^{*}\tilde{u}=(-1)^{q+1}\widetilde{(\overline{\partial}_{b}u)}$, modulo terms that are $O(\|u\|)$. Consequently, a subelliptic estimate or a compactness estimate holds for $q$-forms if and only if the corresponding estimate holds for $(n-1-q)$-forms. In view of the characterization of subellipticity in terms of finite type \cite{Cat83, Cat87}, and of compactness of $N_{q}$ on convex domains by the absence of $q$-dimensional varieties from the boundary (\cite{FuSt98}), such a symmetry between form levels is manifestly absent in the $\overline{\partial}$-Neumann problem. (The analogous construction performed for forms on $\Omega$ yields a $\tilde{u}$ that in general is not in the domain of $\overline{\partial}^{*}$.) At one point, we will actually need a version of the tilde operators that intertwines $\overline{\partial}_{b}$ and $\overline{\partial}_{b}^{*}$ without $0$-th order error terms; we discuss such a construction in an appendix (section \ref{appendix}).

Our results are as follows. First, we prove the analogue for compactness of the fact that subellipticity of $G_{q}$ implies subellipticity of $N_{q}$ (\cite{Dia86, Koe04}). It is worthwhile to note that our method provides, in the case of boundaries of smooth pseudoconvex domains in $\mathbb{C}^{n}$, a new proof of this result as well, compare Remark \ref{alsoforsubelliptic} below.
\begin{thm} \label{thm:cpt bdy implies cpt inside}
Let $\Omega\subset\C^n$ be a bounded pseudoconvex domain with
smooth boundary. Let $1\leq q\leq n-2$. If the complex Green operator $G_{q}$ is a compact operator on $L^2_{(0,q)}(b\Omega)$,  
then the $\dbar$-Neumann operator $N_{q}$ is a compact operator on $L^2_{(0,q)}(\Omega)$.
\end{thm}
Theorem \ref{thm:cpt bdy implies cpt inside} is proved in Section \ref{sec:cpt bdy to int}. Our strategy is simple. Because $\Box_{q}=\overline{\partial}^{*}\overline{\partial}+\overline{\partial}\overline{\partial}^{*}$ acts componentwise as (a constant multiple of) the real Laplacian, the $L^{2}$-norm of a form $u \in \text{dom}(\overline{\partial}) \cap \text{dom}(\overline{\partial}^{*})$ is controlled by $\|\overline{\partial}u\|+\|\overline{\partial}^{*}u\|$ plus the $(-1/2)$-norm of the trace on the boundary. To the tangential part of this trace, one applies the compactness estimate from the assumption in the theorem, estimating the norms of $\overline{\partial}_{b} u_{tan}$ and of $\overline{\partial}_{b}^{*}u_{tan}$ via trace theorems (here $u_{tan}$ denotes the tangential part of $u$; taking the tangential part `loses' the normal component of the form, but this component is benign). In order to avoid various issues related to trace theorems, we actually work in $W^{2}_{(0,q)}(\Omega)$, rather than in $L^{2}_{(0,q)}(\Omega)$. Our arguments involve (as usual) absorbing terms, and since $N_{q}$ is not a priori known to preserve $W^{2}_{(0,q)}(\Omega)$, we use elliptic regularization to ensure finiteness of the terms to be absorbed. We thus obtain compactness of $N_{q}$ (`only') on  $W^{2}_{(0,q)}(\Omega)$. However, as pointed out in \cite{FuSt01}, because $N_{q}$ is self-adjoint in $L^{2}_{(0,q)}(\Omega)$, compactness in $W^{2}_{(0,q)}(\Omega)$ implies compactness in $L^{2}_{(0,q)}(\Omega)$, by a general principle from functional analysis. 
\begin{rem}\label{intNbdry}
\emph{In view of the symmetry for $\overline{\partial}_{b}$ and its absence for $\overline{\partial}$, discussed above, Theorem \ref{thm:cpt bdy implies cpt inside} implies in particular that compactness or subellipticity of $N_{q}$ need not imply the corresponding property for $G_{q}$ when $q>(n-1)/2$. Of course, the appropriate question becomes whether such an implication holds when compactness or subellipticity is assumed for the $\overline{\partial}$-Neumann operator at levels $q$ \emph{and} $(n-1-q)$. As far as the authors know, this is open for both compactness and subellipticity.}
\end{rem}

Next, we show that Catlin's classical sufficient condition for compactness in the $\overline{\partial}$-Neumann problem (\cite{Cat84}), imposed on symmetric form levels (this is essentially dictated by the discussion above), is also sufficient for compactness of the complex Green operator. Let $\I_q = \{J=(j_1,\dots,j_q)\in \N^q: 1\leq j_1< \cdots < j_q\leq n\}$ and
$\Lambda_z^{(0,q)}$ be the space of $(0,q)$-forms at $z$ equipped with
the standard Hermitian metric $|\sum_{J\in\I_q} u_J\, d\z_J|^2 = \sum_{J\in\I_q}|u_J|^2$. When the $q$-tuple $J \notin \mathcal{I}_{q}$, $u_{J}$ is defined in the usual manner by antisymmetry. For a $C^2$-function $\lam(z)$ defined in a neighborhood of $z$, define
\[
H_q[\lam](z,u) = \sum_{K\in\I_{q-1}}\sum_{j,k=1}^n
\frac{\p^2\lam(z)}{\p z_j\p\z_k} u_{jK} \overline{u_{kK}} \; .
\]
\begin{defn} \label{defn:Prop Pt}
$b\Omega$ satisfies $(P_q)$ if for all $M>0$, there exists
$U_M\supset b\Omega$, $\lam_M\in C^2(U_M)$ so that for all 
$z\in U_M$ and $w\in\Lambda_z^{(0,q)}$
\begin{enumerate}
\item $0 \leq \lambda_{M}(z) \leq 1$,

\item $\displaystyle H_q[\lam](z,w) \geq M |w(z)|^2$.
\end{enumerate}
\end{defn}
$(2)$ can be reformulated in two equivalent ways (by standard facts from (multi)linear algebra): $(a)$ the sum of any $q$ (equivalently: the smallest $q$) eigenvalues of $(\partial^{2}\lambda_{M}/\partial z_{j}\partial\overline{z_{k}})_{jk}$ is at least $M$; $(b)$ on any affine subspace of complex dimension $q$ (provided with the inner product from $\mathbb{C}^{n})$, the (real) Laplacian of $\lambda_{M}$ is at least $M$. For $(a)$ this can be seen most easily by working in an orthonormal basis that consists of eigenvectors of $(\partial^{2}\lambda_{M}/\partial z_{j}\partial\overline{z_{k}})_{jk}$. This also gives that $(b)$ implies $(2)$. That $(2)$ implies $(b)$ is an application of the Schur majorization theorem (\cite{HJ85}, Theorem 4.3.26), to the effect that the sum of any $q$ diagonal elements of a Hermitian matrix is at least equal to the sum of the $q$ smallest eigenvalues. Note that if the sum of the smallest $q$ eigenvalues is at least $M$, then so is the sum of the smallest $(q+1)$ (since the additional eigenvalue is necessarily nonnegative). That is, $({P}_{q})$ implies $({P}_{q+1})$ (but not vice versa). Sibony (\cite{Sib87}) studied $(P_q)$ from the point of view of Choquet theory for the cone of functions $\lambda$ with $H_{q}[\lambda] \geq 0$ (actually only for $q=1$, but his arguments work essentially verbatim for $q>1$, see \cite{FuSt01}). This work provides in particular examples of domains with `big' (say of positive measure) sets of points of infinite type in the boundary, whose $\overline{\partial}$-Neumann and complex Green operators are nevertheless compact. 
\begin{thm}\label{thm:Pq implies cptness on bdy}
Let $\Omega\subset \C^n$ be a pseudoconvex domain with $C^\infty$ boundary
and $1\leq q \leq n-2$. 
If $b\Omega$ satisfies  $(P_q)$ and  $(P_{n-1-q})$ (equivalently: $(P_{\tilde{q}})$, where $\tilde{q} = \min \{q, n-1-q\}$), 
then $G_{q}$ and $G_{n-1-q}$ are  compact operators on 
$L^2_{(0,q)}(b\Omega)$ and $L^2_{(0,n-1-q)}(b\Omega)$, respectively.
\end{thm}
We prove Theorem \ref{thm:Pq implies cptness on bdy} in Section \ref{finalproof}. It suffices to produce compact solution operators for $\overline{\partial}_{b}$. To do so, we follow Shaw (\cite{Shaw85}) in representing a $\overline{\partial}_{b}$--closed form $u$ on the boundary as the difference of two $\overline{\partial}$--closed forms, $\alpha^{-}$ on $\Omega$ and $\alpha^{+}$ on the complement: $u=\alpha^{+}-\alpha^{-}$. Then, roughly speaking, $({P}_{q})$ lets us solve the equation $\overline{\partial}\beta^{-}=\alpha^{-}$ on $\Omega$, with suitable compactness estimates, while $({P}_{n-1-q})$ lets us do the same for $\overline{\partial}\beta^{+}=\alpha^{+}$ on an appropriate `annular' region surrounding $\overline{\Omega}$. That the latter can be done follows essentially from work of Shaw in \cite{Shaw85a}. The details are given in Section \ref{annulus} (Proposition \ref{cor:Nq is cpt, all weights}). We mention that in \cite{McN02}, McNeal has introduced a condition called $(\tilde{P}_{q})$ which is implied by $(P_{q})$ and which is sufficient for compactness of $N_{q}$. Whether $(\tilde{P}_{q})$ can take the place of $(P_{q})$ in Theorem \ref{thm:Pq implies cptness on bdy} is open. This has to do with the fact that the exact relationship between $(P)$, $(\tilde{P})$, and compactness is not understood. However, $(P)$ and $(\tilde{P})$ are known to be equivalent on locally convexifiable domains (see the discussion in \cite{Str06}), so that in our next result, $(\tilde{P})$ can take the place of $(P)$. 

Theorem \ref{thm:cpt bdy implies cpt inside}, Theorem \ref{thm:Pq implies cptness on bdy}, and work of Fu and Straube \cite{FuSt98, FuSt01} immediately allow us to characterize compactness of the complex Green operator on smooth bounded locally convexifiable domains. We say that a domain is locally convexifiable if for every boundary point there is a neighborhood, and a biholomorphic map defined on this neighborhood, that takes the intersection of the domain with the neighborhood onto a convex domain.
\begin{thm}\label{thm:convex characterization}
Let $\Omega\subset\C^n$ be a smooth bounded locally convexifiable domain, and let $1\leq q \leq n-2$. Then the following are equivalent:

\noindent $(i)$ The complex Green operator $G_{q}$ is compact.

\noindent $(ii)$ Both $G_{q}$ and $G_{n-1-q}$ are compact.

\noindent $(iii)$ The $\overline{\partial}$-Neumann operators $N_{q}$ and $N_{n-1-q}$ are compact.

\noindent $(iv)$ $b\Omega$ satisfies both $(P_q)$ and $(P_{n-1-q})$.

\noindent $(v)$ $b\Omega$ does not contain (germs of) complex varieties of dimension $q$ nor of dimension $(n-1-q)$.
\end{thm}
\begin{proof}
On a locally convexifiable domain, compactness of $N_{q}$ is \emph{equivalent} to each of $(iv)$ and $(v)$, \emph{at level $q$} (\cite{FuSt98, FuSt01}). In particular, $(iii)$, $(iv)$, and $(v)$ are equivalent on these domains, and by Theorem \ref{thm:Pq implies cptness on bdy}, they imply $(ii)$. $(i)$ and $(ii)$ are equivalent by the symmetry in the form levels for $\overline{\partial}_{b}$. By Theorem \ref{thm:cpt bdy implies cpt inside}., $(ii)$ implies $(iii)$.
\end{proof}

In Theorems \ref{thm:Pq implies cptness on bdy} and \ref{thm:convex characterization} we assume $1 \leq q \leq n-2$, thus excluding the endpoints $q=0$ and $q=(n-1)$. Formally, this restriction arises because if $q=0$ or $q = (n-1)$, then $\min\{q,n-1-q\} = 0$, and it is not clear what an appropriate interpretation of $(P_{0})$ should be. This is analogous to the situation in the interior. However, $N_{0} =  \overline{\partial}^{*}N_{1}^{2}\overline{\partial} = \overline{\partial}^{*}N_{1}(\overline{\partial}^{*}N_{1})^{*}$ (\cite{ChSh01}, Theorem 4.4.3.), and compactness of $N_{1}$ implies that of $N_{0}$. $(P_{1})$ therefore is a sufficient condition for compactness of $N_{0}$. This situation persists on the boundary. Let $n \geq 3$ and assume $b\Omega$ satisfies $(P_{1})$, and hence $(P_{n-2})$. Then $G_{1}$ and $G_{n-2}$ are compact, by Theorem \ref{thm:Pq implies cptness on bdy}. In turn, this implies that both $G_{0}$ and $G_{n-1}$ are compact, by formulas analogous to the one quoted above for $N_{0}$. That is, $(P_{1})$ is a sufficient condition for compactness of $G_{0}$ and $G_{n-1}$. 

{\bf{Acknowledgment}}: We are indebted to Ken Koenig for pointing out the observation in his paper \cite{Koe04} of the symmetry in form levels for $\overline{\partial}_{b}$ with regard to compactness and subellipticity. He also found some oversights and inaccuracies in an earlier version of this manuscript.

%
%
\section{Proof of Theorem \ref{thm:cpt bdy implies cpt inside}}\label{sec:cpt bdy to int}
Let $W^s(U)$ be the usual Sobolev space of order $s$ on $U$ ($U$ may be an open subset of $\mathbb{C}^{n}$ or of $b\Omega$) and let $W^s_{(0,q)}(U)$ be space of $(0,q)$-forms with coefficients in $W^s(U)$. We first express compactness of $G_{q}$ in the usual way in terms of a family of estimates. Let $s \geq 0$. Then for any $\ep>0$, there exists $C_\ep>0$ so that
\begin{equation}\label{compact1}
\|u\|_{W^{s}(b\Omega)}
\leq \ep\big( \|\dbarb u\|_{W^{s}(b\Omega)} + \|\dbarbs u\|_{W^{s}(b\Omega)}\big)
+ C_{\ep} \|u\|_{W^{-1}(b\Omega)} \; ,
\end{equation}
for all $u \in \text{dom}(\overline{\partial}_{b}) \cap \text{dom}(\overline{\partial}_{b}^{*})$ when $1\leq q\leq (n-2)$. When $q= (n-1)$, we assume that $u \perp \text{ker}(\overline{\partial}_{b}^{*})$ (equivalently $u \in \text{Im}(\overline{\partial}_{b})$). When $s=0$, \eqref{compact1} is the standard compactness estimate that is equivalent to compactness of $G_{q}$; the proof is the same as that for the corresponding statement concerning the $\overline{\partial}$-Neumann operator (see for example \cite{FuSt01}, Lemma 1.1). This remark applies likewise to lifting the estimate to higher Sobolev norms. 

The main a priori estimate to be proved is as follows: for all $\ep>0$, there exists $C_\ep>0$ so that if $u\in\Dom(\dbars) \cap W^2_{(0,q)}(\Omega)$, $\dbar u\in W^2_{(0,q+1)}(\Omega), \text{ and }
\dbars u\in W^2_{(0,q-1)}(\Omega)$, then 
\begin{equation}\label{eqn:compactness of domain dbar, dbars in H1}
\| u \|_{W^2(\Omega)} \leq \ep\big(\|\dbar u\|_{W^2(\Omega)}
+ \|\dbars u\|_{W^2(\Omega)}\big)
+ C_{\ep} \|u\|_{W^1(\Omega)} \; .
\end{equation}
Fix a defining function $\rho$ for $\Omega$, so that $|\nabla\rho|=1$ near $b\Omega$ (i.e. take $\rho$ to agree with the signed boundary distance near $b\Omega$). We first estimate the normal component of $u=\sum_{J\in \mathcal{I}_{q}}u_{J}d\overline{z_{J}} \in W^{2}_{(0,q)}(\Omega)\cap \text{dom}(\overline{\partial}^{*})$. This component is given by $u_{norm}=\sum_{j=1}^{n}\sum_{K\in \mathcal{I}_{q-1}}(\partial\rho /\partial z_{j})u_{jK}d\overline{z_{K}}$, and its trace on $b\Omega$ vanishes. Because $\vartheta\overline{\partial}+ \overline{\partial}\vartheta$ acts coefficientwise as a constant multiple of the (real) Laplacian, we obtain 
\begin{multline}\label{u-normal}
\|u_{norm}\|_{W^2(\Omega)} \leq \|\Delta u_{norm}\|_{L^{2}(\Omega)} \\
\leq C \big(\|\dbar u\|_{W^1(\Omega)} + \|\dbars u\|_{W^1(\Omega)} + \|u\|_{W^1(\Omega)}\big) \\
\leq \ep \big(\|\dbar u\|_{W^2(\Omega)} + \|\dbars u\|_{W^2(\Omega)}\big) + 
C_\ep\|u\|_{W^1(\Omega)} \; ;
\end{multline}
it is assumed that $u$ is as in Theorem \ref{thm:cpt bdy implies cpt inside}. The last inequality in \eqref{u-normal} comes from interpolating Sobolev norms ($\|f\|_{W^{1}(\Omega)} \leq \varepsilon\|f\|_{W^{2}(\Omega)}+C_{\varepsilon}\|f\|_{L^{2}(\Omega)}$).
Similarly, we have that $u_{norm} \in W^{3}_{(0,q-1)}(\Omega)$ and 
\begin{equation}\label{u-norm2}
\|u_{norm}\|_{W^3(\Omega)} \leq \|\Delta u_{norm}\|_{W^{1}(\Omega)} 
\leq C \big(\|\dbar u\|_{W^2(\Omega)} + \|\dbars u\|_{W^2(\Omega)} + \|u\|_{W^2(\Omega)}\big) \; .
\end{equation}
\eqref{u-norm2} is important because it implies that $\|\overline{\partial}u_{norm}\|_{W^{2}(\Omega)}$ and $\|\overline{\partial}^{*}u_{norm}\|_{W^{2}(\Omega)}$ are also dominated by the right hand side of \eqref{u-norm2}. (Note that $u_{norm}$ has vanishing trace on the boundary, so it \emph{is} in $\text{dom}(\overline{\partial}^{*})$.) Then so are  $\|\overline{\partial}u_{tan}\|_{W^{2}(\Omega)}$ and $\|\overline{\partial}^{*}u_{tan}\|_{W^{2}(\Omega)}$, where $u_{tan}=u-(\overline{\partial}\rho\wedge u_{norm})$ is the tangential part of $u$. As a result, it now only remains to establish \eqref{eqn:compactness of domain dbar, dbars in H1} for $u_{tan}$.

Let $\ep>0$. As in \eqref{u-normal}, we have
\begin{multline}\label{u-tan1}
\|u_{tan}\|_{W^{2}(\Omega)} 
\approx \|\triangle u_{tan}\|_{L^2(\Omega)}  + \|u_{tan}\|_{W^{3/2}(b\Omega)} \\
\lesssim\big(\|\dbar u_{tan}\|_{W^1(\Omega)} + \|\dbars u_{tan}\|_{W^1(\Omega)}\big)
+ \|u_{tan}\|_{W^{3/2}(b\Omega)} \\
\leq \ep \big(\|\dbar u_{tan}\|_{W^2(\Omega)} + \|\dbars u_{tan}\|_{W^2(\Omega)}\big)
+C_\ep \| u_{tan}\|_{W^1(\Omega)}+\|u_{tan}\|_{W^{3/2}(b\Omega)} \; .
\end{multline}
We are going to apply \eqref{compact1} to the last term on the right hand side of \eqref{u-tan1}. Note that by definition, $\overline{\partial}_{b}(u_{tan}|_{b\Omega})$ is the trace of $\overline{\partial}u_{tan}$, projected onto the tangential $(q+1)$-forms. Because $u_{tan}$ is tangential near the boundary, $\overline{\partial}_{b}^{*}(u_{tan}|_{b\Omega})$ equals the trace of $\overline{\partial}^{*}u_{tan}$ on the boundary, modulo a term of order zero (i.e. involving no derivatives of $u_{tan}$). Therefore, with $s=3/2$ in \eqref{compact1},
\begin{multline}\label{u-tan2}
\|u_{tan}\|_{W^{3/2}(b\Omega)}
\leq \ep \big(\|\dbarb (u_{tan})\|_{W^{3/2}(b\Omega)}
+ \|\dbarbs (u_{tan})\|_{W^{3/2}(b\Omega)}\big)
+ C_{\ep} \|u_{tan}\|_{W^{-1}(b\Omega)} \\
\leq \ep \big(\|\dbar u_{tan}\|_{W^{3/2}(b\Omega)} 
+ \|\dbars u_{tan}\|_{W^{3/2}(b\Omega)} + \|u_{tan}\|_{W^{3/2}(b\Omega)}\big)
+ C_{\ep} \| u_{tan}\|_{W^{1/2}(b\Omega)} \\
\leq \ep \big(\|\dbar u_{tan}\|_{W^2(\Omega)}
+ \|\dbars u_{tan}\|_{W^2(\Omega)} + \|u_{tan}\|_{W^{2}(\Omega)}\big)
+ C_{\ep} \|u_{tan}\|_{W^{1}(\Omega)} \; .
\end{multline} 
Putting \eqref{u-tan2} into the right hand side of \eqref{u-tan1} (for $\|u_{tan}\|_{W^{3/2}(b\Omega)}$) and absorbing $\varepsilon\|u_{tan}\|_{W^{2}(\Omega)}$ gives \eqref{eqn:compactness of domain dbar, dbars in H1} for $u_{tan}$. With this, by what was said above, \eqref{eqn:compactness of domain dbar, dbars in H1} is established for $u$.

Because \eqref{eqn:compactness of domain dbar, dbars in H1} is only an a priori estimate, and $N_{q}u$ is not (yet) known to be in $W^{2}_{(0,q)}(\Omega)$ for $u \in W^{2}_{(0,q)}(\Omega)$, we work first with the regularized operators $N_{\delta,q}$, $0<\delta \leq 1$, arising from elliptic regularization (\cite{FoKo72}, 2.3, \cite{Ta96}, 12.5). $N_{\delta,q}$ inverts (in $L^{2}_{(0,q)}(\Omega)$) the operator $\Box_{\delta,q}$, the unique self-adjoint operator associated to the quadratic form $Q_{\delta}(u,v)=(\overline{\partial}u,
\overline{\partial}v)_{L^{2}(\Omega)}+(\overline{\partial}^{*}u,
\overline{\partial}^{*}v)_{L^{2}(\Omega)}+\delta(\nabla u,\nabla v)_{L^{2}(\Omega)}$, with form domain $W^{1}_{(0,q)}(\Omega) \cap \text{dom}(\overline{\partial}^{*})$. Equivalently: for $u \in L^{2}_{(0,q)}(\Omega)$, $v \in W^{1}_{(0,q)}(\Omega) \cap \text{dom}(\overline{\partial}^{*})$, $(u,v)_{L^{2}(\Omega)}=Q_{\delta}(N_{\delta,q}u,v)$. Note that for $u \in \text{dom}(\Box_{\delta,q})$, 
\begin{equation}\label{boxdelta}
\Box_{\delta,q}u=((-1/4)+\delta)\Delta u \; ,
\end{equation} 
where $\Delta$ acts coefficientwise. (Nonetheless, $\Box_{\delta,q}$ is \emph{not} a multiple of $\Box_{q}$; the domain has changed.) $Q_{\delta}(u,u)$ is coercive (it dominates $\|u\|_{W{1}(\Omega)}$), and consequently $N_{\delta,q}$ gains two derivatives in Sobolev norms (see e.g. \cite{Ta96}, 12.5). In particular, when $u \in C^{\infty}_{(0,q)}(\overline{\Omega})$, then so is $N_{\delta,q}u$.

We now claim that the $N_{\delta,q}$ are compact on $W^{2}_{(0,q)}(\Omega)$, `uniformly' in $\delta >0$. That is, we claim the following uniform compactness estimate: for every $\varepsilon >0$, there exists a constant $C_{\varepsilon}$ such that for $u \in W^{2}_{(0,q)}(\Omega)$ and $0<\delta\leq 1$,
\begin{equation}\label{Ndelta1}
\|N_{\delta, q} u\|_{W^2(\Omega)}^2 \leq \ep \|u\|_{W^2(\Omega)}^2 + C_\ep \|u\|_{L^2(\Omega)}^2 \; .
\end{equation}
(This type of estimate is equivalent to compactness, compare \cite{McN02}, Lemma 2.1; \cite{DAngelo02}, Proposition~V.2.3$\,$.) Because $C^{\infty}_{(0,q)}(\overline{\Omega})$ is dense in $W^{2}_{(0,q)}(\Omega)$ (and $N_{\delta,q}$ is continuous in $W^{2}_{(0,q)}(\Omega)$), it suffices to establish \eqref{Ndelta1} for $u \in C^{\infty}_{(0,q)}(\overline{\Omega})$. So let $u \in C^{\infty}_{(0,q)}(\overline{\Omega})$. In the following estimates, all constants will be uniform in $\delta$. We first apply \eqref{eqn:compactness of domain dbar, dbars in H1} to $N_{\delta,q}u$:
\begin{equation}\label{Ndelta2}
\|N_{\delta,q}u\|_{W^{2}(\Omega)} \leq \varepsilon (\|\overline{\partial}N_{\delta,q}u\|_{W^{2}(\Omega)}+\|\overline{\partial}^{*}N_{\delta,q}u\|_{W^{2}(\Omega)}) + C_{\varepsilon}\|N_{\delta,q}u\|_{W^{1}(\Omega)} \; .
\end{equation}
Interpolating Sobolev norms, using that $N_{\delta,q}$ is bounded in $L^{2}_{(0,q)}(\Omega)$ with a bound that is uniform in $\delta$, and absorbing $\varepsilon\|N_{\delta,q}U\|_{W^{2}(\Omega)}$, \eqref{Ndelta2} gives
\begin{equation}\label{Ndelta3}
\|N_{\delta,q}u\|_{W^{2}(\Omega)} \leq \varepsilon (\|\overline{\partial}N_{\delta,q}u\|_{W^{2}(\Omega)}+\|\overline{\partial}^{*}N_{\delta,q}u\|_{W^{2}(\Omega)}) + C_{\varepsilon}\|u\|_{L^{2}(\Omega)}\; .
\end{equation}
The estimate $\|\overline{\partial}N_{q}u\|_{W^{2}(\Omega)}+\|\overline{\partial}^{*}N_{q}u\|_{W^{2}(\Omega)} \lesssim \|N_{q}u\|_{W^{2}(\Omega)}+\|u\|_{W^{2}(\Omega)}$ is standard; we next show that it remains valid for the regularized operators $N_{\delta,q}$, with constants uniform in $\delta$. This is known, but somewhat hard to pinpoint in the literature. 

By interior elliptic regularity (uniform in $\delta >0$, in view of \eqref{boxdelta}) (and interpolation of Sobolev norms), we can estimate $\varepsilon(\|\overline{\partial}N_{\delta,q}u\|_{W^{2}(U)} + \|\overline{\partial}^{*}N_{\delta,q}u\|_{W^{2}(U)})$ by the right hand side of \eqref{Ndelta1} for any relatively compact subdomain $U$ of $\Omega$. As a result, it suffices to estimate $\varphi\overline{\partial}N_{\delta,q}u$ and
$\varphi\overline{\partial}^{*}N_{\delta,q}u$ for a smooth cutoff function compactly supported in a special boundary chart. We will also let differential operators act coefficientwise in the associated special boundary frame. Then, a tangential derivative will preserve the domain of $\overline{\partial}^{*}$. (For information on special boundary charts and frames, the reader may consult \cite{FoKo72} or \cite{ChSh01}.) We denote by $\nabla_T$ the gradient with respect to the tangential variables
and by $\p/\p\nu$ the normal derivative.

For the tangential derivatives, we have
\begin{multline}\label{Ndelta4}
\|\nabla_{T}^{2}\varphi\overline{\partial}N_{\delta,q}u\|_{L^{2}(\Omega)}^{2}
+ \|\nabla_{T}^{2}\varphi\overline{\partial}^{*}N_{\delta,q}u\|_{L^{2}(\Omega)}^{2} \\
\lesssim \|\varphi\overline{\partial}\nabla_{T}^{2}N_{\delta,q}u\|_{L^{2}(\Omega)}^{2}
+ \|\varphi\overline{\partial}^{*}\nabla_{T}^{2}N_{\delta,q}u\|_{L^{2}(\Omega)}^{2} + \|N_{\delta,q}u\|_{W^{2}(\Omega)}^{2} \\
\lesssim Q_{\delta}(\nabla_{T}^{2}N_{\delta,q}u,\nabla_{T}^{2}N_{\delta,q}u) + \|N_{\delta,q}u\|_{W^{2}(\Omega)}^{2} \;\;\;\; \\
\lesssim |(\nabla_{T}^{2}u,\nabla_{T}^{2}N_{\delta,q}u)| + \|N_{\delta,q}u\|_{W^{2}(\Omega)}^{2} \\
\lesssim  \|N_{\delta,q}u\|_{W^{2}(\Omega)}^{2} + \|u\|_{W^{2}(\Omega)}^{2} \; . \;\;\;\;\;\;\;\;\;\;\;\;\;\;\;\;\;\;\;\;\;\;\;\;\;\;\;\;\;\;\;\;\;\;\;\;\;\;\;\;\;\;\;\;
\end{multline}
Here, constants are allowed to depend on $\varphi$ (but not on $\delta$). We have used the estimate $Q_{\delta}(\nabla_{T}^{2}N_{\delta,q}u,\nabla_{T}^{2}N_{\delta,q}u)          \lesssim |(\nabla_{T}^{2}u,\nabla_{T}^{2}N_{\delta,q}u)| + \|N_{\delta,q}u\|_{W^{2}(\Omega)}^{2}$, which follows from \cite{KoNi65}, Lemma 3.1, \cite{FoKo72}, Lemma 2.4.2$\;$. In the case at hand, it can be established by the usual procedure: repeated integration by parts and commuting operators as necessary to make terms of the form $Q_{\delta}(N_{\delta,q}u,v)=(u,v)$ appear.

We now come to the normal derivatives. Expressing the real Laplacian in the coordinates of the special boundary chart gives
\begin{equation}\label{eqn:2 normal derivs}
\Big\| \frac{\p^2 v}{\p\nu^2} \Big\|_{L^2(\Omega)}^2
\leq C_{q,\Omega}\Big( \|\triangle v\|_{L^2(\Omega)}^2
+ \| v \|_{W^1(\Omega)}^2 + \|\nabla_T^2 v \|_{L^2(\Omega)}^2\Big)
\end{equation}
Note that $\| \triangle \varphi \dbar N_{\delta,q} u  \|_{L^2(\Omega)}^2
\approx \|\varphi \overline{\partial}\Delta N_{\delta,q} u  \|_{L^2(\Omega)}^2 + \|\overline{\partial}N_{\delta,q}u\|_{W^{1}(\Omega)}
\lesssim \|u\|_{W^1(\Omega)}^2 + \|N_{\delta,q}u\|_{W^{2}(\Omega)}$, in view of \eqref{boxdelta}. There is a similar estimate for the Laplacian of $\varphi\dbars N_{\delta,q} u$. Applying (\ref{eqn:2 normal derivs}) to $v = \varphi\dbar N_{\delta,q} u$ and $\varphi\dbars N_{\delta,q} u$ respectively, and using \eqref{Ndelta4} for the $\nabla_{T}^{2}$ terms now shows that
\begin{equation}\label{Ndelta5}
\Big\| \frac{\p^2 (\varphi\dbar N_{\delta,q} u)}{\p\nu^2} \Big\|_{L^2(\Omega)}^2 
+ \Big\| \frac{\p^2 (\varphi\dbars N_{\delta,q} u)}{\p\nu^2} \Big\|_{L^2(\Omega)}^2
 \lesssim \|N_{\delta,q}u\|_{W^{2}(\Omega)}^{2}+\|u\|_{W^{2}(\Omega)}^{2} \;.
\end{equation}

Concerning the $L^{2}$-norms of the mixed derivatives, we note that they are dominated by the $L^{2}$-norms of the pure derivatives (\cite{LiMa72}, Theorem 7.4), and therefore by the right hand sides of \eqref{Ndelta4} and \eqref{Ndelta5}. (We remark that using this fact is merely a convenience, not a necessity; normal derivatives can be expressed in terms of tangential ones, $\overline{\partial}$, $\vartheta$, and terms of order zero.) Together with \eqref{Ndelta3}, \eqref{Ndelta4}, and \eqref{Ndelta5}, this gives (via a partition of unity subordinate to a cover of the boundary by special boundary charts, and summing over the charts plus a compactly supported term)
\begin{equation}\label{Ndelta6}
\|N_{\delta,q}u\|_{W^{2}(\Omega)} \leq \varepsilon (\|N_{\delta,q}u\|_{W^{2}(\Omega)}+ \|u\|_{W^{2}(\Omega)}) + C_{\varepsilon}\|u\|_{L^{2}(\Omega)} \; ,
\end{equation}
where the family $C_{\varepsilon}$ has been rescaled. After absorbing $\varepsilon \|N_{\delta,q}u\|_{W^{2}(\Omega)}$, \eqref{Ndelta6} gives \eqref{Ndelta1} (first for $\varepsilon$ less than $1/2$, say; but that is sufficient).

Because the constant $C_{\varepsilon}$ in \eqref{Ndelta1} is independent of $\delta$, we can let $\delta$ tend to zero and obtain \eqref{Ndelta1} with $N_{q}u$. Indeed, if $u\in W^2_{(0,q)}(\Omega)$, then $\{ N_{\delta,q} u : 0<\delta\leq 1\}$ is a bounded set in $W^2_{(0,q)}(\Omega)$. So there exists a sequence $\delta_n\to0$ and $\hat u\in W^2_{(0,q)}(\Omega)$ 
so that $N_{\delta_n,q} u \rightarrow  u$ weakly in $W^2_{(0,q)}(\Omega)$.
One easily checks that $\hat{u} \in \text{dom}(\overline{\partial}^{*})$. Also, if $v\in W^1_{(0,q)}(\Omega)\cap\Dom(\dbars)$, then $\lim_{n\to\infty} Q_{\delta_n}(N_{{\delta_n},q}u,v) = Q(\hat u,v)$. However, $Q_{\delta_n}(N_{\delta_n,q}u,v) = (u,v) = Q(N_q u,v)$. Thus, $Q(\hat u,v)=
Q(N_q u,v)$ for all $v \in W^{1}_{(0,q)}(\Omega) \cap \text{dom}(\overline{\partial}^{*})$, whence $\hat u = N_q u$ (since $W^{1}_{(0,q)}(\Omega) \cap \text{dom}(\overline{\partial}^{*})$ is dense in $\text{dom}(\overline{\partial}) \cap \text{dom}(\overline{\partial}^{*})$ with respect to the graph norm induced by $Q$). Consequently, $N_q u$ is in $W^2_{(0,q)}(\Omega)$ and satisfies \eqref{Ndelta1}, and so is a compact operator on $W^2_{(0,q)}(\Omega)$.

It remains to be seen that $N_{q}$ is compact on $L^{2}_{(0,q)}(\Omega)$. This turns out to be a consequence of its compactness on $W^{2}_{(0,q)}(\Omega)$ and a general fact from functional analysis which we now state. Suppose $H$ is a Hilbert space and $B \subset H$ is a dense subspace. Assume that $B$ is provided with a norm under which it is complete, and under which it embeds continuously into $H$ (i.e. $\|u\|_{H} \leq C\|u\|_{B}$ for $u \in B$). Suppose $T$ is a bounded linear operator on $B$ which is symmetric with respect to the inner product induced from $H$: $(Tu,v)_{H}=(u,Tv)_{H}$ for all $u,v \in B$. If $T$ is compact on $B$, then it (has a unique extension to $H$ which) is compact on $H$. This is Corollary II from \cite{Lax54}. In our situation, it suffices to take $H = L^{2}_{(0,q)}(\Omega)$, $B=W^{2}_{(0,q)}(\Omega)$, and $T=N_{q}$, to conclude that $N_{q}$ is compact on $L^{2}_{(0,q)}(\Omega$). 

This concludes the proof of Theorem \ref{thm:cpt bdy implies cpt inside}. 

\begin{rem}\label{alsoforsubelliptic}
\emph{Our proof above,combined with \cite{Str95}, yields a new proof that subellipticity of $G_{q}$ implies subellipticity of $N_{q}$, $1 \leq q \leq (n-2)$. We give an outline. Assume $G_{q}$ is subelliptic of order $2s$. Equivalently: $\|u\|_{W^{s}(b\Omega)} \lesssim  \|\overline{\partial}_{b}u\|_{L^{2}(b\Omega)}+\|\overline{\partial}_{b}^{*}u\|_{L^{2}(b\Omega)}$ for $u \in \text{dom}(\overline{\partial}_{b}) \cap \text{dom}(\overline{\partial}_{b}^{*})$. First, our arguments above can be followed almost verbatim to obtain that $N_{q}$ maps $W^{2-s}_{(0,q)}(\Omega)$ to $W^{2+s}_{(0,q)}(\Omega)$. One change needed is in (the analogue of) \eqref{Ndelta4}: the inner product $(\nabla_{T}^{2}u,\nabla_{T}^{2}N_{q}u)$ has to be estimated by $|(\nabla_{T}^{2}u,\nabla_{T}^{2}N_{q}u)| \lesssim
\|\nabla_{T}^{2}u\|_{W^{-s}(\Omega)}\|\nabla_{T}^{2}N_{q}u\|_{W^{s}(\Omega)} \lesssim \|u\|_{W^{2-s}(\Omega)}\|N_{q}u\|_{W^{2+s}(\Omega)} \leq (s.c.)\|N_{q}u\|_{W^{2+s}(\Omega)}^{2}+(l.c.)\|u\|_{W^{2-s}(\Omega)}^{2}$. Note that a simplification occurs in that we no longer have to work with the regularized operators: $N_{q}$ is already known to be compact (by Theorem \ref{thm:cpt bdy implies cpt inside}), hence to be globally regular. That this translates into honest subellipticity of $N_{q}$ follows from arguments used in \cite{Str95} in a closely related context. These arguments are also based on \cite{Lax54}. (Very) roughly speaking, they are as follows. Denote by $\Lambda^{s}$ the standard tangential Bessel potential operators of order $s$ (alternatively: the $s$-th power of the tangential Laplace-Beltrami operator may be used). Then $N_{q}$ mapping $W^{2-s}_{(0,q)}(\Omega)$ (continuously) to $W^{2+s}_{(0,q)}(\Omega)$ is equivalent to $\Lambda^{s}N_{q}\Lambda^{s}$ (we are omitting cutoff functions) being continuous on $W^{2}_{(0,q)}(\Omega)$. But $\Lambda^{s}N_{q}\Lambda^{s}$ is symmetric with respect to the $L^{2}$ inner product, and so is then also continuous in $L^{2}_{(0,q)}(\Omega)$, by \cite{Lax54}, Theorem I. (Theorem I is considerably more elementary than Corollary II used above for compactness.) As a result, $N_{q}$ maps $W^{-s}_{(0,q)}(\Omega)$ continuously to $W^{s}_{(0,q)}(\Omega)$. By interpolation with the continuity from $W^{2-s}_{(0,q)}(\Omega)$ to $W^{2+s}_{(0,q)}(\Omega)$, $N_{q}$ maps $L^{2}_{(0,q)}(\Omega)$ continuously to $W^{2s}_{(0,q)}(\Omega)$ (see e.g. \cite{LiMa72}, Theorem 12.4, for the interpolation between $W^{-s}(\Omega)$ and $W^{2-s}(\Omega)$, negative indices require some care).}
\end{rem}

%
%
\section{Compactness of the $\dbar$-Neumann operator on an `annulus'}\label{annulus}
In this section we prove an auxiliary result that will be used in the proof of Theorem \ref{thm:Pq implies cptness on bdy}, but which is of independent interest. If $\Omega$ and $\Omega_{1}$ are two bounded pseudoconvex domains with $\overline{\Omega} \subset \Omega_{1}$, we call $\Omega^{+} := \Omega_{1} \setminus \overline{\Omega}$ an `annulus'. The $\dbar$-Neumann problem on $\Omega^+$ has been studied in \cite{Shaw85a}. Since $\Omega^+$ is not pseudoconvex, $\ker(\Box_{q})$ need not be trivial, and we let
\[
\H_{q} = \ker(\Box_{q}) = \{ u\in \Dom(\dbar)\cap\Dom(\dbars) : \dbar u =0,\ \dbars u =0\},
\]
the harmonic $(0,q)$-forms. Let $H_{q}$ be the orthogonal projection of $L^2_{(0,q)}(\Omega^+)$ onto $\H_{q}$. 
\begin{prop}\label{cor:Nq is cpt, all weights}
Let $\Omega^{+}$ be an `annulus' as above, with smooth boundary, let $1 \leq q \leq (n-2)$. Assume the outer boundary of $\Omega^{+}$ satisfies $(P_{q})$, and the inner boundary satisfies $({P}_{n-1-q})$. Then $\H_{q}$ is finite dimensional and $N_{q}$ is compact. 
\end{prop}
We comment on the assumptions. $({P}_{n-1-q})$ arises as follows. When establishing compactness of $N_{q}$ assuming a condition  like $({P})$ (on a pseudoconvex domain), one uses the Kohn-Morrey-H\"{o}rmander formula, or a twisted version of it (\cite{McN02, Str06}). In the case of an `annulus' between two pseudoconvex domains, the part of the boundary integral from the inner boundary has the wrong sign (it is nonpositive instead of nonnegative), and a modification is needed. From the work in \cite{Shaw85a}, when applied to our situation, it turns out that the condition needed on the Hessian of a suitable function is precisely that the sum of any $(n-1-q)$ eigenvalues be at least $M$. The details are as follows. 

In order to state Shaw's result, we work temporarily in a special boundary chart. Let $L_1, \dots, L_{n-1}$ be a (local) orthonormal basis of $T^{(1,0)}(b\Omega)$, $L_{n}$ the complex (unit) normal, and  $\omega_j$ the $(1,0)$-form dual to $L_j$. For a function $f$, let
$f_{jk}$ be defined by $\p\dbar f = \sum_{j,k=1}^n f_{jk}\, \omega_j\wedge\overline{\omega_k}$. Let $\varphi \in C^{\infty}(\overline{\Omega^{+}})$, real valued, and let $u = \sum_{J \in \mathcal{I}_{q}}u_{J}\overline{\omega_{J}} \in C^{\infty}_{(0,q)}(\overline{\Omega^{+}}) \cap \text{dom}(\overline{\partial}_{\varphi}^{*})$, supported in a special boundary chart for the \emph{inner} boundary. Here, $\overline{\partial}_{\varphi}^{*}$ is the adjoint of $\overline{\partial}$ with respect to the weighted $L^{2}$ inner product with weight $e^{-\varphi}$. The computations that lead to $(3.23)$ in \cite{Shaw85a} are valid for general weight functions, and there is the following analogue of $(3.23)$ (this is made explicit in \cite{Ah07}, Proposition 2.1):
\begin{multline}\label{ann1}
\sum_{K\in\I_{q-1}}\sum_{j,k=1}^{n} \int_{\Omega^+} \vp_{jk} u_{jK}\overline{u_{kK}}
e^{-\vp}\, dV 
- \sum_{J\in\I_{q}}\int_{\Omega^+}\Big(\sum_{j=1}^{n-1} \vp_{jj}\Big)|u_J|^2 e^{-\vp}\, dV \\
+ \sum_{K\in\I_{q-1}}\sum_{j,k=1}^{n} \int_{b\Omega} \rho_{jk} u_{jK}\overline{u_{kK}}
e^{-\vp}\, d\sigma 
- \sum_{J\in\I_{q}}\int_{b\Omega}\Big(\sum_{j=1}^{n-1} \rho_{jj}\Big)|u_J|^2 e^{-\vp}\, d\sigma \\
+\sum_{J \in \I{q}}\left (\|\L_n u_J\|_\vp^2 + \sum_{j=1}^{n-1}\|\delta_j u_J\|_\vp^2 \right )\\
\leq C\big(\|\dbar u\|_\vp^2 + \|\dbars_\vp u\|_\vp^2 + \|u\|_\vp^2\big)\; ,\;\;\;\;\;\;\;\;\;\;\;\;\;\;\;\;\;\;\;\;\;\;\;\;\;\;\;\;\;\;\; \; 
\end{multline}
where $\delta_{j}=e^{\varphi}L_{j}e^{-\varphi}$, $\rho$ is a defining function for $\Omega^{+}$, and $C$ is independent of $\varphi$. Denote by $\mathcal{I}_{q}^{\prime}$ the set of strictly increasing $q$-tuples that do not contain $n$. In the second line of \eqref{ann1}, the sums are effectively only over $1 \leq k,j \leq (n-1)$, $K \in \mathcal{I}^{\prime}_{q-1}$, and $J \in \mathcal{J}^{\prime}_{q}$, respectively ($u_{nK}=0$ on $b\Omega$; $u \in \text{dom}(\overline{\partial}^{*})$). The integrand in this line (without the weight factor) is therefore
\begin{equation}\label{ann1a}
\sum_{K\in\I_{q-1}^{\prime}}\sum_{j,k=1}^{n-1} \rho_{jk} u_{jK}\overline{u_{kK}}
- \sum_{J\in\I_{q}^{\prime}}\Big(\sum_{j=1}^{n-1}\rho_{jj}\Big) |u_J|^{2}
= \sum_{K\in\I_{q-1}^{\prime}}\sum_{j,k=1}^{n-1} \left (\rho_{jk} 
- \frac 1q \Big(\sum_{\ell=1}^{n-1} \rho_{\ell\ell}\Big)\delta_{jk}\right )
u_{jK}\overline{u_{kK}} \; .
\end{equation}
This is because every $|u_{J}|^{2}$ can be written in precisely $q$ ways as $|u_{jK}|^{2}$. Note that the Hessian of $\rho$ is negative semidefinite on the complex tangent space at points of $b\Omega \subset b\Omega^{+}$. As a result, the second line in \eqref{ann1} is nonnegative: the right hand side equals at least $|u|^{2}$ times the sum of the smallest $q$ eigenvalues of the Hermitian matrix
\[ \left ( \rho_{jk}- \frac{1}{q}\left (\sum_{j=1}^{n-1}\rho_{jj} \right )\delta_{jk} \right)_{j,k=1}^{n-1} \; \] 
(this is analogous to the discussion of property $(P_{q})$ in section 1). Such a sum equals minus the trace $\sum_{j=1}^{n-1} \rho_{jj}$ plus a sum of $q$ eigenvalues of $(\rho_{jk})_{j,k=1}^{n-1}$, hence is at least equal to the negative of the sum of the largest $(n-1-q)$ eigenvalues of $((\rho)_{jk})_{j,k=1}^{n-1}$, and so is nonnegative. Letting the sums in the first line of \eqref{ann1} run only over $K \in \mathcal{I}_{q-1}^{\prime}$, $J \in \mathcal{I}_{q}^{\prime}$, and $1 \leq j.k \leq n-1$, respectively, makes a mistake that involves (coefficients of) the normal component of $u$. These terms can be estimated by the right hand side of \eqref{ann1} plus $C_{\varphi}\|e^{-\varphi /2}u\|^{2}_{W_{-1}(\Omega^{+})}$. This follows from an argument similar to that in \eqref{u-normal}, starting with the $W^{1}$ to $W^{-1}$ version of the first inequality in \eqref{u-normal}, and using interpolation of Sobolev norms to get rid of the dependence on $\varphi$ of the (first) constant. Therefore, estimate \eqref{ann1} remains valid when the sums in the first line are restricted so that no normal components of $u$ appear, and the right hand side is augmented by $C_{\varphi}\|e^{-\varphi /2}u\|^{2}_{W_{-1}(\Omega^{+})}$. Observe that as in \eqref{ann1a}, the integrand (without the weight factor $e^{-\varphi}$) in the first line in \eqref{ann1} is then
\begin{equation}\label{ann2}
\sum_{K\in\I_{q-1}^{\prime}}\sum_{j,k=1}^{n-1} \vp_{jk} u_{jK}\overline{u_{kK}}
- \sum_{J\in\I_{q}^{\prime}}\Big(\sum_{j=1}^{n-1}\vp_{jj}\Big) |u_J|^{2}
= \sum_{K\in\I_{q-1}^{\prime}}\sum_{j,k=1}^{n-1} \left (\vp_{jk} 
- \frac 1q \Big(\sum_{\ell=1}^{n-1} \vp_{\ell\ell}\Big)\delta_{jk}\right )
u_{jK}\overline{u_{kK}} \; ,
\end{equation}
where $\delta_{jk}$ denotes the Kronecker $\delta$. Omitting the nonnegative second and third lines from \eqref{ann1} (in its modified form), we obtain for $u$ supported in a special boundary chart:
\begin{multline}\label{ann3}
\sum_{K\in\I_{q-1}^{\prime}}\sum_{j,k=1}^{n-1}\int_{\Omega^{+}} \left (\varphi_{jk} 
- \frac 1q \Big(\sum_{\ell=1}^{n-1} \vp_{\ell\ell}\Big)\delta_{jk}\right )
u_{jK}\overline{u_{kK}}\,e^{-\varphi}dV \\
\leq C\big(\|\dbar u\|_\vp^2 + \|\dbars_\vp u\|_\vp^2 + \|u\|_\vp^2\big)
+ C_{\varphi}\|e^{-\vp/2} u\|^{2}_{W^{-1}(\Omega)}\; .
\end{multline}
We are now ready to prove Proposition \ref{cor:Nq is cpt, all weights}.

\begin{proof}[Proof of Proposition \ref{cor:Nq is cpt, all weights}] 
Fix $M$ (sufficiently large) and choose a function $\lambda_{M} \in C^{\infty}(\overline{\Omega^{+}})$, $0 \leq \lambda_{M} \leq 1$, that agrees near the inner boundary (say, in $U_{1} \cap \Omega^{+}$) with $-\mu_{M}$, where $\mu_{M}$ is a function given by the definition of $(P_{n-1-q})$, and near the outer boundary (say, on $U_{2} \cap \Omega^{+}$) with a function given by the definition of $(P_{q})$. We claim that we have the following estimate (for $M \geq M_{0}$):
\begin{equation}\label{ann4}
\|u\|_{\lambda_{M}}^{2} \leq \frac{C}{M}(\|\overline{\partial}u\|_{\lambda_{M}}^{2} + \|\overline{\partial}_{\lambda_{M}}^{*}u\|_{\lambda_{M}}^{2}) + C_{M}\|u\|_{W^{-1}(\Omega^{+})}^{2} \; , u \in \text{dom}(\overline{\partial}) \cap \text{dom}(\overline{\partial}^{*}) \; ,
\end{equation}
with a constant $C$ that does not depend on $M$. Note that saying that $u \in \text{dom}(\overline{\partial}^{*})$ is the same as saying that $u \in \text{dom}(\overline{\partial}_{\lambda_{M}}^{*})$. Also, the unweighted norms and the weighted norms are equivalent, with bounds that are uniform in $M$ (because $0 \leq \lambda_{M} \leq 1$).

It suffices to establish \eqref{ann4} for forms that are smooth up to the boundary; the density of these forms in $\text{dom}(\overline{\partial}) \cap \text{dom}(\overline{\partial}^{*})$ does not require pseudoconvexity (\cite{Hor65}, Proposition 2.1.1; \cite{ChSh01}, Lemma 4.3.2). 

Assume first that $u$ is supported near the outer boundary, on $U_{2} \cap \overline{\Omega^{+}}$. Then \eqref{ann4} follows immediately from the Kohn-Morrey-H\"{o}rmander formula (\cite{Hor65}, Proposition 2.1.2; \cite{ChSh01}, Proposition 4.3.1) and from $(2)$ in the definition of $(P_{q})$. 

Now assume that $u$ is similarly supported near the inner boundary, on $U_{1} \cap \overline{\Omega^{+}}$. Via a partition of unity, we may assume that $u$ is supported in a special boundary chart. We use \eqref{ann3} with $\varphi = \lambda_{M}$. Note that $\lambda_{M}=-\mu_{M}$ near the support of $u$, where $\mu_{M}$ satisfies $(1)$ and $(2)$ in the definition of $(P_{n-1-q})$. At a point, the integrand on the left hand side of \eqref{ann3} (without the exponential factor) is at least as big as $|u|^{2}$ times the sum of the smallest $q$ eigenvalues of the Hermitian matrix 
\begin{equation}\label{ann5}
\left ((-\mu_{M})_{jk} 
+ \frac 1q \Big(\sum_{\ell=1}^{n-1} (\mu_{M})_{\ell\ell}\Big)\delta_{jk}\right )_{j,k=1}^{n-1} \; 
\end{equation}
(see again the discussion of property $({P}_{q})$ in section 1). Such a sum equals the trace $\sum_{\ell=1}^{n-1} (\mu_{M})_{\ell\ell}$ minus a sum of $q$ eigenvalues of $((\mu_{M})_{jk})_{j,k=1}^{n-1}$, hence is at least equal to the sum of the smallest $(n-1-q)$ eigenvalues of $((\mu_{M})_{jk})_{j,k=1}^{n-1}$, which in turn is at least equal to the sum of the smallest $(n-1-q)$ eigenvalues of $((\mu_{M})_{jk})_{j,k=1}^{n}$ (by the equivalence of $(2)$ and (b), or directly by the Schur majorization theorem (\cite{HJ85}, Theorem 4.3.26). That is, the sum is at least equal to the sum of the smallest $(n-1-q)$ eigenvalues of $(\partial^{2}\mu_{M} / \partial z_{j}\partial\overline{z_{k}})_{j,k=1}^{n}$, so is at least equal to $M$. This gives \eqref{ann4} (after absorbing the term $C\|u\|^{2}_{\lambda_{M}}$ and rescaling $M$), but with $u$ on the left hand side replaced by the tangential part of $u$. Again, the normal component is under control, and as in \eqref{ann1}, the square of its norm is estimated by the right hand side of \eqref{ann4}. This proves estimate \eqref{ann4} when $u$ is supported near the inner boundary.

When $u$ has compact support in $\Omega^{+}$, 
\eqref{ann4} follows by interior elliptic regularity of $\overline{\partial}\oplus\overline{\partial}^{*}_{\lambda_{M}}$, with a constant $C$ that is independent of the support ($C_{M}$ depends on the support). The reason that $C$ may be taken to be independent of the support is that $\|\overline{\partial}u\| + \|\overline{\partial}_{\lambda_{M}}^{*}u\|$ controls the $W^{1}$-norm on a relatively compact subset. Since we only need to bound the $L^{2}$-norm, we can interpolate between the $W^{1}$-norm and the $W^{-1}$-norm to get the desired estimate.

Finally, when $u$ is general, choose a partition of unity on $\overline{\Omega^{+}}$, $\chi_{0}, \chi_{1}$, and $\chi_{2}$, such that $\chi_{0}$ is compactly supported in $\Omega^{+}$, and $\chi_{1}$ and $\chi_{2}$ are supported in $U_{1}$ and $U_{2}$, respectively. We then have \eqref{ann4} for $\chi_{0}u$, $\chi_{1}u$, and $\chi_{2}u$. The right hand sides of these estimates contain terms where $\overline{\partial}$ or $\overline{\partial}^{*}_{\lambda_{M}}$ produce derivatives of the cutoff functions. However, these terms contain no derivatives of $u$, and they are compactly supported. Consequently, they can be estimated as in the previous paragraph. Collecting the resulting estimates establishes \eqref{ann4}, with $C$ independent of $M$.

\eqref{ann4} implies that from every sequence $\{u_{k}\}_{k=1}^{\infty}$ in $\text{dom}(\overline{\partial}) \cap \text{dom}(\overline{\partial}_{\lambda_{M}}^{*})$ (which equals $\text{dom}(\overline{\partial}) \cap \text{dom}(\overline{\partial}^{*})$) with $\|u_{n}\|_{\lambda_{M}}$ bounded and $\overline{\partial}u_{k} \rightarrow 0$, $\overline{\partial}_{\lambda_{M}}^{*}u_{k} \rightarrow 0$, one can extract a subsequence which converges in (weighted) $L^{2}_{(0,q)}(\Omega^{+})$. It suffices to find a subsequence which converges in $W^{-1}_{(0,q)}(\Omega^{+})$ (using that $L^{2}_{(0,q)}(\Omega^{+}) \hookrightarrow W^{-1}_{(0,q)}(\Omega^{+})$ is compact); \eqref{ann4} implies that such a subsequence is Cauchy (hence convergent) in $L^{2}_{(0,q)}(\Omega^{+})$. General Hilbert space theory (see \cite{Hor65}, Theorems 1.1.3 and 1.1.2) now gives that $\text{ker}(\Box_{\lambda_{M},q})$ is finite dimensional and that $\overline{\partial}: L^{2}_{(0,q)}(\Omega^{+}) \rightarrow  L^{2}_{(0,q+1)}(\Omega^{+})$ and $\overline{\partial}_{\lambda_{M}}^{*}: L^{2}_{(0,q)}(\Omega^{+}) \rightarrow  L^{2}_{(0,q-1)}(\Omega^{+})$ have closed range. But then $\overline{\partial}: L^{2}_{(0,q-1)}(\Omega^{+}) \rightarrow L^{2}_{(0,q)}(\Omega^{+})$ also has closed range (its adjoint in the weighted space has closed range), and consequently, so does $\overline{\partial}^{*}$ (acting on $(0,q)$-forms; this also follows from the formula $\overline{\partial}^{*}v=e^{-\lambda_{M}}\overline{\partial}_{\lambda_{M}}^{*}(e^{\lambda_{M}}v)$). Therefore, we have the estimate
\begin{equation}\label{ann6}
\|u\|_{L^{2}(\Omega^{+})} \lesssim \|\overline{\partial}u\|_{L^{2}(\Omega^{+})} + \|\overline{\partial}^{*}u\|_{L^{2}(\Omega^{+})} +\|H_{q}u\|_{L^{2}(\Omega^{+})} \;
\end{equation}
for $u \in \text{dom}(\overline{\partial}) \cap \text{dom}(\overline{\partial}^{*})$. This estimate implies the existence of $N_{q}$ as a bounded operator on $L^{2}_{(0,q)}(\Omega^{+})$ that inverts $\Box_{q}$ on $\mathcal{H}_{q}^{\perp}$ (see for example \cite{Shaw85a}, Lemma 3.2 and its proof). Moreover, the range of $\overline{\partial}: L^{2}_{(0,q-1)}(\Omega^{+}) \rightarrow L^{2}_{(0,q)}(\Omega^{+})$ has finite codimension in $\text{ker}(\overline{\partial}) \subset L^{2}_{(0,q)}(\Omega^{+})$, because $\text{ker}(\Box_{\lambda_{M},q})$ is finite dimensional). But the (unweighted) orthogonal complement of this range in $\text{ker}(\overline{\partial}) \subset L^{2}_{(0,q)}(\Omega^{+})$ equals $\text{ker}(\Box_{q})$, which is therefore finite dimensional as well. 

To see that $N_{q}$ is compact, it suffices to show compactness on $\mathcal{H}_{q}^{\perp}$ (since $N_{q}$ is zero on $\mathcal{H}_{q}$). When $u \in \mathcal{H}_{q}^{\perp}$, we have from \eqref{ann6} (since $N_{q}u \in \mathcal{H}_{q}^{\perp}$)
\begin{equation}\label{ann7}
\|N_{q}u\|_{L^{2}(\Omega^{+})} \lesssim \|\overline{\partial}N_{q}u\|_{L^{2}(\Omega^{+})} + \|\overline{\partial}^{*}N_{q}u\|_{L^{2}(\Omega^{+})} = \|(\overline{\partial}^{*}N_{q+1})^{*}u\|_{L^{2}(\Omega^{+})} + \|\overline{\partial}^{*}N_{q}u\|_{L^{2}(\Omega^{+})} \; .
\end{equation}
Therefore, we only need to show that both $\overline{\partial}^{*}N_{q}$ and  $\overline{\partial}^{*}N_{q+1}$ are compact. Now $\overline{\partial}^{*}N_{q+1}\alpha$ gives the norm minimizing solution to $\overline{\partial}v=\alpha$ when $\alpha \in \text{Im}(\overline{\partial}) \subset L^{2}_{(0,q+1)}(\Omega^{+})$, while $\overline{\partial}_{\lambda_{M}}^{*}N_{\lambda_{M},q+1}\alpha$ gives a different solution (the one that minimizes the weighted norm). For such $\alpha$, \eqref{ann4} therefore implies (with constants independent of $M$)
\begin{multline}\label{ann8}
\|\overline{\partial}^{*}N_{q+1}\alpha\|_{L^{2}(\Omega^{+})}^{2} \leq \|\overline{\partial}_{\lambda_{M}}^{*}N_{\lambda_{M},q+1}\alpha\|_{L^{2}(\Omega^{+})}^{2} \lesssim 
\|\overline{\partial}_{\lambda_{M}}^{*}N_{\lambda_{M},q+1}\alpha\|_{\lambda_{M}}^{2} \\
\lesssim \frac{C}{M}\|\alpha\|_{\lambda_{M}}^{2}+ C_{M}\|\overline{\partial}_{\lambda_{M}}^{*}N_{\lambda_{M},q+1}\alpha\|_{W^{-1}(\Omega^{+})}^{2} \\
\lesssim \frac{C}{M}\|\alpha\|_{L^{2}(\Omega^{+})}^{2}+ C_{M}\|\overline{\partial}_{\lambda_{M}}^{*}N_{\lambda_{M},q+1}\alpha\|_{W^{-1}(\Omega^{+})}^{2} \; .
\end{multline}
Because $C$ is independent of $M$ and $\overline{\partial}_{\lambda_{M}}^{*}N_{\lambda_{M},q+1}: L^{2}_{(0,q+1)}(\Omega^{+}) \rightarrow W^{-1}_{(0,q)}(\Omega^{+})$ is compact ($L^{2}(\Omega^{+})$ imbed compactly into $W^{-1}(\Omega^{+})$), \eqref{ann8} implies that $\overline{\partial}^{*}N_{q+1}$ is compact on $\text{Im}(\overline{\partial})$ (\cite{McN02}, Lemma 2.1, \cite{DAngelo02}, Proposition V.2.3). But on the orthogonal complement of $\text{Im}(\overline{\partial})$, $\overline{\partial}^{*}N_{q+1} =0$, and so $\overline{\partial}^{*}N_{q+1}$ is compact from $L^{2}_{(0,q+1)}(\Omega^{+}) \rightarrow L^{2}_{(0,q)}(\Omega^{+})$. To estimate $\overline{\partial}^{*}N_{q}\alpha$, we cannot invoke \eqref{ann4} directly (because $\overline{\partial}^{*}N_{q}\alpha$ is a $(q-1)$-form), and an additional step is needed. We have (again for $\alpha \in \text{Im}(\overline{\partial}) \subset L^{2}_{(0,q)}(\Omega^{+})$)
\begin{multline}\label{ann9}
\|\overline{\partial}_{\lambda_{M}}^{*}N_{\lambda_{M},q}\alpha\|_{\lambda_{M}}^{2} 
= (\overline{\partial}\overline{\partial}_{\lambda_{M}}^{*}N_{\lambda_{M},q}\alpha, N_{\lambda_{M},q}\alpha)_{\lambda_{M}} =(\alpha, N_{\lambda_{M},q}\alpha)_{\lambda_{M}} \\
\leq \frac{2C}{M}\|\alpha\|_{\lambda_{M}}^{2} + \frac{M}{2C}\|N_{\lambda_{M},q}\alpha\|_{\lambda_{M}}^{2} \\
\leq \frac{2C}{M}\|\alpha\|_{\lambda_{M}}^{2} + \frac{1}{2}\|\overline{\partial}_{\lambda_{M}}^{*}N_{\lambda_{M},q}\alpha\|_{\lambda_{M}}^{2}
+ C_{M}\|N_{\lambda_{M}}\alpha\|_{W^{-1}(\Omega^{+})}^{2} \; .
\end{multline}
Here we have used that $\overline{\partial}\alpha = 0$ and that $\alpha \perp_{\lambda_{M}} \mathcal{H}_{\lambda_{M},q}$ (since $\alpha \in \text{Im}(\overline{\partial})$) in the equality in the second line, the inequality $|ab| \leq (1/A)a^{2}+Ab^{2}$, and \eqref{ann4} for the last estimate. The middle term in the last line can now be absorbed, and combining the resulting estimate with $\|\overline{\partial}^{*}N_{q}\alpha\|_{L^{2}(\Omega^{+})}^{2} \leq \|\overline{\partial}_{\lambda_{M}}^{*}N_{\lambda_{M},q}\alpha\|_{L^{2}(\Omega^{+})}^{2}$ gives an analogue of \eqref{ann8}. The rest of the argument is the same as above. This completes the proof of Proposition \ref{cor:Nq is cpt, all weights}.
\end{proof}

%
%
%
%
\section{Property $(P)$ and compactness of the complex Green operator}\label{finalproof}
 
In this section, we prove Theorem \ref{thm:Pq implies cptness on bdy}. We may assume that $1\leq q \leq n-1-q$, and we must show that $G_{q}$ is compact (by the symmetry between form levels discussed in section 1, $G_{n-1-q}$ is then compact as well). For $u \in L^{2}_{(0,q)}(b\Omega)$, we have the Hodge decomposition $u=\overline{\partial}_{b}\overline{\partial}_{b}^{*}G_{q}u +  \overline{\partial}_{b}^{*}\overline{\partial}_{b}G_{q}u$ (\cite{ChSh01}, Theorem 9.4.2). In particular, when $\overline{\partial}_{b}u = 0$, $\overline{\partial}_{b}^{*}G_{q}u$ gives the solution of minimal $L^{2}$-norm (the canonical solution) to the equation $\overline{\partial}_{b}\alpha = u$. $G_{q}$ can be expressed in terms of these canonical solution operators at levels $q$ and $q+1$ and their adjoints (\cite{BoSt91}, p. 1577):
\begin{equation}\label{proof1}
G_{q} = (\overline{\partial}_{b}^{*}G_{q})^{*}(\overline{\partial}_{b}^{*}G_{q}) + 
(\overline{\partial}_{b}^{*}G_{q+1})(\overline{\partial}_{b}^{*}G_{q+1})^{*} \; .
\end{equation}
This formula (including the proof) is analogous to the corresponding formula for the $\overline{\partial}$-Neumann operator (\cite{FoKo72, Ran84}). Therefore, compactness of $G_{q}$ is equivalent to compactness of both $\overline{\partial}_{b}^{*}G_{q}$ and $\overline{\partial}_{b}^{*}G_{q+1}$, and we shall prove the latter. Since projection onto the orthogonal complement of $\text{ker}(\overline{\partial}_{b})$ preserves compactness, we only have to produce \emph{some} solution with suitable estimates.  

We first consider $\overline{\partial}_{b}^{*}G_{q}$. Choose a ball $B$ so that $\overline{\Omega} \subset\subset B$. Set $\Omega^{+} = B\setminus \overline{\Omega}$. Let $\alpha \in \text{ker}(\overline{\partial}_{b}) \cap C^{\infty}_{(0,q)}(b\Omega)$. By \cite{ChSh01}, Lemma 9.3.5., there exist \emph{$\overline{\partial}$-closed forms} $\alpha^{+} \in C^{1}_{(0,q)}(\overline{\Omega^{+}}) \subset W^{1}_{(0,q)}(\Omega^{+})$ and $\alpha^{-} \in C^{1}_{(0,q)}(\overline{\Omega}) \subset W^{1}_{(0,q)}(\Omega)$ such that
\begin{equation}\label{jump}
\alpha = \alpha^{+} - \alpha^{-} \;\; \text{on} \;\, b\Omega \; 
\end{equation}
(in the sense of traces of the coefficients, but also in the sense of restrictions of forms; i.e. the normal components of $\alpha^{+}$ and $\alpha^{-}$ cancel each other out at points of $b\Omega$). Moreover
\begin{equation}\label{jump1}
\|\alpha^{+}\|_{W^{1/2}(\Omega^{+})} \lesssim \|\alpha\|_{L^{2}(b\Omega)} 
\end{equation}
and
\begin{equation}\label{jump2}
\|\alpha^{-}\|_{W^{1/2}(\Omega)} \lesssim \|\alpha\|_{L^{2}(b\Omega)} \; .
\end{equation}
Estimates \eqref{jump1} and \eqref{jump2} are from \cite{ChSh01}, Lemma 9.3.6. However, a comment is in order, as only \eqref{jump2} is explicit there, while the estimate for $\alpha^{+}$ is given in terms of a $\overline{\partial}$-closed continuation of $\alpha^{+}$ to all of $B$. On $B$, only the $(-1/2)$-norm of this continuation can be estimated. But this loss occurs across $b\Omega$, and so does not affect the estimate on $\Omega^{+}$. \eqref{jump1} is implicit in \cite{ChSh01}; we sketch the argument (which is the same as for $\alpha^{-}$). 

$\alpha^{+}$ is of the form $\alpha^{+} = (1/2)(E_{k}\alpha - V_{k})$, where the notation is the same as in \cite{ChSh01}. $E_{k}\alpha$ is an extension of $\alpha$, based on an extension operator $E$ for functions (i.e. the coefficients), modified so that $\overline{\partial}E_{k}\alpha$ vanishes to order $k$ on $b\Omega$ (\cite{ChSh01}, (9.3.12a) and (9.3.12b)). From the definition of $E$ (\cite{ChSh01}, (9.3.8)), it is easily checked that $\|E_{k}\alpha\|_{W^{1/2}(\Omega^{+})} \lesssim \|\alpha\|_{L^{2}(b\Omega)}$. $V_{k}$ is obtained as the solution of a $\overline{\partial}$ problem on $B$. It has the form $V_{k} = \overline{\partial}^{*}N_{q+1}^{B}\tilde{U}_{k}$, where $\tilde{U}_{k}$ is a $(q+1)$-form supported in the intersection $\Omega^{+}_{\delta}$ of a thin tubular neighborhood $\Omega_{\delta}$ of $b\Omega$ with $\Omega^{+}$ ($\Omega_{\delta}$ is thin enough so that the usual tangential Sobolev norms make sense). $\tilde{U}_{k}$ satisfies the estimate $\|\tilde{U}_{k}\|_{W^{-1/2}(B)} \lesssim  |||\tilde{U}_{k}|||_{W^{-1/2}(\Omega_{\delta})} \lesssim \|\alpha\|_{L^{2}(b\Omega)}$; the first inequality is dual to $|||u|||_{W^{1/2}(\Omega_{\delta})} \lesssim \|u\|_{W^{1/2}(\Omega_{\delta})}$, the second is (9.3.15) in \cite{ChSh01}. Because $\overline{\partial}^{*}N_{q+1}^{B}$ locally gains a full derivative, we obtain that $V_{k}$ is in $W^{1/2}_{(0,q)}$ away from the outer boundary of $\Omega^{+}$ (i.e. the boundary of $B$). Near the outer boundary, $(1/2)$-estimates follow from the pseudo-local estimates for $\overline{\partial}^{*}N_{q+1}^{B}$ near the boundary of $B$ (taking $\|\tilde{U}_{k}\|_{W^{-1/2}(B)}$ as the weak global term). Putting these estimates together gives that $\|V_{k}\|_{W^{1/2}(\Omega^{+})} \lesssim \|\alpha\|_{L^{2}(b\Omega)}$, and \eqref{jump1} follows.

$\Omega^{+}$ is not pseudoconvex, so $\overline{\partial}\alpha^{+}=0$ does not automatically imply that $\alpha^{+}$ is in the range of $\overline{\partial}$. That it is follows from \cite{ChSh01}: $\alpha^{+}$ has a $\overline{\partial}$-closed extension to $B$, which is in the range of $\overline{\partial}$ on $B$. By restriction, $\alpha^{+}$ is in the range of $\overline{\partial}$ on $\Omega^{+}$. Therefore, we have on $\Omega^{+}$:
\begin{equation}\label{jump3}
\alpha^{+} = \overline{\partial}\overline{\partial}^{*}N^{\Omega^{+}}_{q}\!\!\alpha^{+} \; .
\end{equation}
Similarly, on $\Omega$,
\begin{equation}\label{jump4}
\alpha^{-} = \overline{\partial}\overline{\partial}^{*}N^{\Omega}_{q}\alpha^{-} \; .
\end{equation}
Note that both $\beta^{+} :=\overline{\partial}^{*}N^{\Omega^{+}}_{q}\alpha^{+}$ and $\beta^{-} :=\overline{\partial}^{*}N^{\Omega}_{q}\alpha^{-}$ have vanishing normal component on $b\Omega$ (as elements of $\text{dom}(\overline{\partial}^{*})$). We also use here that compactness of $N^{\Omega^{+}}_{q}$ (from Proposition \ref{cor:Nq is cpt, all weights}) and of $N^{\Omega}_{q}$ (because $b\Omega$ satisfies $(P_{q})$) lift to higher Sobolev norms (see \cite{KoNi65}, Theorems 2 and 2$^{\prime}$, in particular the remark at the end of the proof of Theorem 2 (page 466)). In particular, $\beta^{+}$ and $\beta^{-}$ are in $W^{1}$ of the respective domains (since $\alpha^{+} \in C^{1}(\overline{\Omega^{+}})$, $\alpha^{-} \in C^{1}(\overline{\Omega})$), and so have traces on $b\Omega$. We obtain that on $b\Omega$
\begin{equation}\label{jump5}
\alpha = \overline{\partial}_{b}\beta^{+} - \overline{\partial}_{b}\beta^{-} = \overline{\partial}_{b}(\beta^{+} - \beta^{-}) \; ,
\end{equation}
where we also use $\beta^{+}$ and $\beta^{-}$ to denote the traces on the boundary. The elliptic theory for the (real) Laplacian ( see for example \cite{LiMa72}) gives
\begin{multline}\label{jump6}
\|\beta^{+}\|_{L^{2}(b\Omega)} \leq \|\beta^{+}\|_{L^{2}(b\Omega^{+})} \lesssim \|\beta^{+}\|_{W^{1/2}(\Omega^{+})} + \|\Delta\beta^{+}\|_{W^{-1}(\Omega^{+})} \\
= \|\overline{\partial}^{*}N^{\Omega^{+}}_{q}\alpha^{+}\|_{W^{1/2}(\Omega^{+})} + \|\Delta\overline{\partial}^{*}N^{\Omega^{+}}_{q}\alpha^{+}\|_{W^{-1}(\Omega^{+})} \\
\lesssim \|\overline{\partial}^{*}N^{\Omega^{+}}_{q}\alpha^{+}\|_{W^{1/2}(\Omega^{+})}
+ \|\alpha^{+}\|_{L^{2}(\Omega^{+})} \; .
\end{multline}
Here $\Delta$ acts coefficientwise on forms. In the last estimate, we have used that $\Box_{q}$ also acts as $\Delta$ coefficientwise on forms (up to a constant factor). The corresponding estimate for $\beta^{-}$ is 
\begin{equation}\label{jump7}
\|\beta^{-}\|_{L^{2}(b\Omega)} \lesssim \|\overline{\partial}^{*}N^{\Omega}_{q}\alpha^{-}\|_{W^{1/2}(\Omega)}
+ \|\alpha^{-}\|_{L^{2}(\Omega)} \; .
\end{equation}
Combining \eqref{jump5} with estimates \eqref{jump6} and \eqref{jump7}, we obtain for the canonical solution operator $\overline{\partial}_{b}^{*}G_{q}$:
\begin{equation}\label{jump8}
\|\overline{\partial}_{b}^{*}G_{q}\alpha\|_{L^{2}(b\Omega)} \lesssim \|\overline{\partial}^{*}N^{\Omega^{+}}_{q}\alpha^{+}\|_{W^{1/2}(\Omega^{+})}
+ \|\alpha^{+}\|_{L^{2}(\Omega^{+})} + \|\overline{\partial}^{*}N^{\Omega}_{q}\alpha^{-}\|_{W^{1/2}(\Omega)}
+ \|\alpha^{-}\|_{L^{2}(\Omega)} \; .
\end{equation}
Both $N^{\Omega^{+}}_{q}$ and $N^{\Omega}_{q}$ are compact ($\Omega^{+}$ satisfies the assumptions in Proposition \ref{cor:Nq is cpt, all weights}). 
Therefore, $\overline{\partial}^{*}N^{\Omega^{+}}_{q}$ and $\overline{\partial}^{*}N^{\Omega}_{q}$ are compact in $W^{1/2}_{(0,q)}(\Omega^{+})$ and $W^{1/2}_{(0,q)}(\Omega)$, respectively (again from \cite{KoNi65}).The embeddings $W^{1/2}(\Omega^{+}) \hookrightarrow L^{2}(\Omega^{+})$ and $W^{1/2}(\Omega) \hookrightarrow L^{2}(\Omega)$ are also compact. \eqref{jump8} was derived for $\alpha \in C^{\infty}_{(0,q)}(b\Omega)$. But $C^{\infty}_{(0,q)}(b\Omega) \cap \text{ker}(\overline{\partial}_{b})$ is dense in $\text{ker}(\overline{\partial}_{b})$ (\cite{ChSh01}, Lemma 9.3.8). In view of \eqref{jump1} and \eqref{jump2}, \eqref{jump8} therefore implies that $\overline{\partial}_{b}^{*}G_{q}$ maps bounded sets in $\text{ker}(\overline{\partial}_{b}) \subset L^{2}_{(0,q)}(b\Omega)$ into relatively compact sets in $L^{2}_{(0,q-1)}(b\Omega)$. In other words, $\overline{\partial}_{b}^{*}G_{q}$ is compact on $\text{ker}(\overline{\partial}_{b})$, hence on $L^{2}_{(0,q)}(b\Omega)$.

We now consider $\overline{\partial}_{b}^{*}G_{q+1}$. $b\Omega$ also satisfies $(P_{q+1})$ (because $(P_{q} \Rightarrow (P_{q+1})$). Assume first that $2q \leq n-2$. Then $b\Omega$ also satisfies $(P_{n-1-(q+1)}) = (P_{n-2-q})$ (since $q \leq (n-2-q)$). Consequently, the previous case applies (with $q$ replaced by $(q+1)$), and $\overline{\partial}_{b}^{*}G_{q+1}$ is compact. Since we may assume without loss of generality that $q \leq (n-1-q)$, i.e. $2q \leq (n-1)$, in proving Theorem \ref{thm:Pq implies cptness on bdy}, the only case left to consider is $2q = (n-1)$. We argue as follows. $(\overline{\partial}_{b}^{*}G_{q})^{*}$, the canonical solution operator to $\overline{\partial}_{b}^{*}$ (as an operator from $(0,q)$-forms to $(0,q-1)$-forms), is compact because $\overline{\partial}_{b}^{*}G_{q}$ is. Because $q-1 = n-1-(q+1)$, the symmetry discussed in section \ref{appendix} yields a compact solution operator for $\overline{\partial}_{b}$ (as an operator from $(0,q)$-forms to $(0,q+1)$-forms). Therefore, the canonical solution operator $\overline{\partial}_{b}^{*}G_{q+1}$ is compact.

The proof of Theorem \ref{thm:Pq implies cptness on bdy} is complete.

\section{Appendix}\label{appendix}

In the last step above, we need a version of Koenig's tilde operators that intertwines $\overline{\partial}_{b}$ and $\overline{\partial}_{b}^{*}$ without the $0$-th order error term that occurs in \cite{Koe04}. A reference for precisely the statement we need seems to be hard to pinpoint in the literature, and we give a brief discussion of a suitable construction. No originality is claimed. Let $\Omega$ be a smooth bounded pseudoconvex domain in $\mathbb{C}^{n}$. For $0 \leq q \leq (n-1)$, define $T_{q}$: $\mathcal{L}^{2}_{(0,q)}(b\Omega)
\rightarrow \mathcal{L}^{2}_{(0,n-1-q)}(b\Omega)$ via
\begin{equation}\label{app1}
\int_{b\Omega}\beta\wedge\alpha\wedge dz_{1}\wedge \cdots \wedge
dz_{n} = (\beta, T_{q}\alpha)_{\mathcal{L}^{2}_{(0,n-1-q)}(b\Omega)}
\; ,
\end{equation}
for all $\beta \in \mathcal{L}^{2}_{(0,n-1-q)}(b\Omega)$. $T_{q}$ is conjugate linear and continuous. 

For $q \leq (n-2)$, $\alpha \in C^{\infty}_{(0,q)}(b\Omega)$, and $\beta \in C^{\infty}_{(0,n-q-2)}(b\Omega)$, we have
\begin{equation}\label{app2}
(\beta,
T_{q+1}\overline{\partial}_{b}\alpha)_{\mathcal{L}^{2}_{(0,n-q-2)}(b\Omega)}
=
\int_{b\Omega}\beta\wedge\overline{\partial}_{b}\alpha\wedge dz_{1}
\cdots\wedge dz_{n} \; .
\end{equation}
In the integral on the right hand side of \eqref{app2}, we can replace $\overline{\partial}_{b}\alpha$ by $d\alpha$ (the extra terms, when wedged with $dz_{1}\wedge \cdots \wedge dz_{n}$, vanish on $b\Omega$). Now integrate by parts to obtain
\begin{equation}\label{app3}
(-1)^{n-1-q}(\beta,
T_{q+1}\overline{\partial}_{b}\alpha)_{\mathcal{L}^{2}_{(0,n-2-q)}(b\Omega)} = (\overline{\partial}_{b}\beta,
T_{q}\alpha)_{\mathcal{L}^{2}_{(0,n-1-q)}(b\Omega)} \; .
\end{equation}
Because $C^{\infty}(b\Omega)$ is dense in $\text{dom}(\overline{\partial}_{b})$ in the graph norm (by Friedrichs' Lemma, see e.g. \cite{ChSh01}, Lemma D.1),
\eqref{app3} holds for $\alpha$ and $\beta$ in $\text{dom}(\overline{\partial}_{b})$ of the respective form levels.
It then follows that $T_{q}\alpha \in \text{dom}(\overline{\partial}_{b}^{*})$, and that
\begin{equation}\label{app4}
(-1)^{n-1-q}T_{q+1}\overline{\partial}_{b}\alpha =
\overline{\partial}_{b}^{*}T_{q}\alpha \; .
\end{equation}

We next compute $T_{q}$ in a special boundary chart. Let $\beta = \sum_{M \in \mathcal{I}_{n-1-q}}\beta_{M}\overline{\omega_{M}}$, $u =\sum_{J \in \mathcal{I}_{q}}u_{J}\overline{\omega_{J}}$. Then (from \eqref{app1})
\begin{equation}\label{app5}
(\beta,T_{q}u) = \sum_{\stackrel{M \in \mathcal{I}_{n-1-q}}{\stackrel{J \in \mathcal{I}_{q}}{\,}}}\int_{b\Omega}\beta_{M}\overline{\omega_{M}}\wedge u_{J}\overline{\omega_{J}}\wedge dz_{1}\wedge \cdots \wedge dz_{n} \; .
\end{equation}
Define the function $h$ by $\overline{\omega_{1}}\wedge \cdots \wedge\overline{\omega_{n-1}}\wedge dz_{1}\wedge \cdots \wedge dz_{n} = hd\sigma$ on $b\Omega$. Then \eqref{app5} gives
\begin{equation}\label{app6}
(\beta,T_{q}u) = \left (\beta, \overline{h}\!\!\!\! \sum_{\stackrel{M \in \mathcal{I}_{n-1-q}}{\stackrel{J \in \mathcal{I}_{q}}{\,}}}\!\!\!\!\varepsilon^{MJ}_{(1, \cdots, n-1)}\overline{u_{J}} \;\overline{\omega_{M}}\right ) \; ,
\end{equation}
i.e.
\begin{equation}\label{app7}
T_{q}u = \overline{h}\!\!\!\! \sum_{\stackrel{M \in \mathcal{I}_{n-1-q}}{\stackrel{J \in \mathcal{I}_{q}}{\,}}}\!\!\!\!\varepsilon^{MJ}_{(1, \cdots, n-1)}\overline{u_{J}}\; \overline{\omega_{M}} \; ,
\end{equation}
where $\varepsilon^{MJ}_{(1, \cdots, n-1)}$ are the usual generalized Kronecker symbols. Then 
\begin{multline}\label{app8}
T_{n-1-q}T_{q}\alpha \;\,=\;\, \overline{h}\!\!\!\!\sum_{\stackrel{M \in \mathcal{I}_{q}}{\stackrel{J \in \mathcal{I}_{n-1-q}}{\,}}}\!\!\!\!\varepsilon^{MJ}_{(1, \cdots, n-1)}\overline{(T_{q}\alpha)_{J}}\;\,\overline{\omega_{M}} \\
= \;\,|h|^{2}\!\!\!\!\sum_{\stackrel{M,K \in \mathcal{I}_{q}}{\stackrel{J \in \mathcal{I}_{n-1-q}}{\,}}}\!\!\!\!\varepsilon^{MJ}_{(1, \cdots, n-1)}\varepsilon^{JK}_{(1, \cdots, n-1)}\overline{\alpha_{K}}\;\,\overline{\omega_{M}}\\ 
=\;\, 
|h|^{2}\sum_{M \in \mathcal{I}_{q}}(-1)^{q(n-1-q)}\overline{\alpha_{M}}\;\,\overline{\omega_{M}} \;\, =\;\,
(-1)^{q(n-1-q)}|h|^{2}\alpha \; .
\end{multline}
Note that $|h|$ is a (nonzero) constant that is globally defined: $\overline{\omega_{1}}\wedge \cdots \wedge\overline{\omega_{n-1}}\wedge dz_{1}\wedge \cdots \wedge dz_{n} = a\,\overline{\omega_{1}}\wedge \cdots \wedge\overline{\omega_{n-1}}\wedge\omega_{1}\wedge \cdots \wedge\omega_{n} = a(const.)(*\omega_{n}) = a (const.) d\sigma$ (when pulled back to $b\Omega$), with $|a|=1$ and $*$ the usual Hodge $*$-operator, see for example Lemma 3.3 and Corollary 3.5 in chapter III of \cite{Ran86}. Thus \eqref{app8} provides $T_{q}^{-1}$ with $T_{q}^{-1} = (const.) \;T_{n-1-q}$. In particular, $T_{q}$ is an isomorphism between the respective spaces. Moreover, \eqref{app4} gives $\|\overline{\partial}_{b}\alpha\| \approx \|\overline{\partial}_{b}^{*}T_{q}\alpha\|$ and, when combined with the fact that $T_{n-1-q}T_{q}\alpha=(const.)\alpha$, also $\|\overline{\partial}_{b}^{*}\alpha\| \approx \|\overline{\partial}_{b}T_{q}\alpha\|$.

Now let $\alpha \in Im(\overline{\partial}_{b}) \subseteq \mathcal{L}^{2}_{(0,q)}(b\Omega)$. Then, by \eqref{app4}, $T_{q}\alpha \in Im(\overline{\partial}_{b}^{*})$, i.e. 
\begin{equation}
T_{q}\alpha = \overline{\partial}_{b}^{*}\overline{\partial}_{b}G_{n-1-q}T_{q}\alpha \; .
\end{equation}
Set $\gamma = (-1)^{n-1-q}T_{q-1}^{-1}\overline{\partial}_{b}G_{n-1-q}T_{q}\alpha$. Then (from \eqref{app4})
\begin{equation}
T_{q}\overline{\partial}_{b}\gamma = (-1)^{n-1-q}\overline{\partial}_{b}^{*}T_{q-1}\gamma = \overline{\partial}_{b}^{*}\overline{\partial}_{b}G_{n-1-q}T_{q}\alpha = T_{q}\alpha \; ,
\end{equation}
that is, $\overline{\partial}_{b}\gamma = \alpha$. So if the canonical solution operator to $\overline{\partial}_{b}^{*}$, $\overline{\partial}_{b}G_{n-1-q}$, is compact, we have produced a compact solution operator for $\overline{\partial}_{b}$ at the symmetric level. Consequently, the canonical solution operator at this level, $\overline{\partial}_{b}^{*}G_{q}$, is compact.

\vspace{0.4in}

\section*{References}

\renewcommand{\refname}{}    
\vspace*{-36pt}              

\providecommand{\bysame}{\leavevmode\hbox to3em{\hrulefill}\thinspace}

\end{document}